\documentclass{amsart}

\usepackage{mathrsfs}

\usepackage[dvips]{graphicx}
\usepackage{amsmath}
\usepackage[latin1]{inputenc}
\usepackage{amsfonts}
\usepackage{amssymb}
\usepackage{graphicx,epsfig}
\usepackage{amsthm}

\usepackage{amscd}
\usepackage{dsfont}
\usepackage[all,dvips]{xy}
\usepackage{fancyhdr}
\usepackage[T1]{fontenc}

\input cyracc.def
\font\tencyr=wncyr10
\def\cyr{\cyracc\tencyr}

\allowdisplaybreaks

\begin{document}

\title{Quantum Quasi-Shuffle Algebras}

\author{Run-Qiang Jian}

\address{D\'{e}partement de Math\'{e}matiques, Universit\'{e} Paris
Diderot (Paris 7), 175, rue du Chevaleret, 75013, Paris, France}
\email{jian@math.jussieu.fr}

\address{Department of Mathematics, Sun Yat-Sen University,
135, Xingang Xi Road, 510275, Guangzhou, P. R. China}

\curraddr {Department of Mathematics, DongGuan University of
Technology, 1, Daxue Road, Songshan Lake, 523808, Dongguan, P. R.
China}
\thanks{}

\author{Marc Rosso}
\address{D\'{e}partement de Math\'{e}matiques, Universit\'{e} Paris
Diderot (Paris 7), 175, rue du Chevaleret, 75013, Paris, France}
\email{rosso@math.jussieu.fr}
\author{Jiao Zhang}
\address{Department of Mathematics, East China Normal
University, Shanghai, P. R. China}
\curraddr{D\'{e}partement de Math\'{e}matiques, Universit\'{e}
Paris Diderot (Paris 7), 175, rue du Chevaleret, 75013, Paris,
France} \email{zhangjiao@math.jussieu.fr}
\thanks{}


\date{}


\keywords{Quantum quasi-shuffle algebra, connected twisted
Yang-Baxter bialgebra, Lyndon word.}\maketitle

\begin{abstract}
We establish some properties of quantum quasi-shuffle algebras.
They include the  necessary  and sufficient condition for the
construction of the quantum quasi-shuffle product, the universal
property, and the commutativity condition. As an application, we
use the quantum quasi-shuffle product to construct a linear basis
of $T(V)$,  for a special kind of Yang-Baxter algebras
$(V,m,\sigma)$.
\end{abstract}

\newtheorem{theorem}{Theorem}[]
\newtheorem{lemma}[theorem]{Lemma}
\newtheorem{proposition}[theorem]{Proposition}
\newtheorem{definition}[theorem]{Definition}
\newtheorem{corollary}[theorem]{Corollary}
\newtheorem{remark}[theorem]{Remark}

\section{Introduction}
Quasi-shuffle algebras are a generalization of shuffle algebras.
They first arose in \cite{NR} for the study of the cofree
irreducible Hopf algebra built on an associative algebra. There,
K. Newman and D. E. Radford constructed an associative algebra
structure on $T(U)$, for an algebra $U$, by combining the
multiplication of $U$ and the shuffle product of $T(U)$. These
algebras have their particular interest in many branches of
algebra and a number of applications have been found in the past
decade. For example, they are used in the study of commutative
TriDendriform algebras \cite{Lod}, Rota-Baxter algebras \cite{EG},
and multiple zeta values \cite{Hof}.

After the birth of quantum groups, many algebraic objects were
better understood in the more general framework of  braided
categories. For example, shuffle algebras, special examples of
quasi-shuffle algebras, had been quantized in \cite{Ro} ten years
ago, and led to a more intrinsic understanding of quantum
enveloping algebras. The next task was to find a suitable way to
quantize the quasi-shuffle algebra. There were some attempts, for
example, \cite{B} and \cite{Hof}. For a braided vector space
$(V,\sigma)$, in order to study all associative algebra structures
on $T(V)$ which are compatible with the "twisted" deconcatenation
coproduct, \cite{JR} introduced the notion of quantum
$B_\infty$-algebras. The quantum $B_\infty$-algebra provides a
suitable framework for the quantization of quasi-shuffle algebras
in the spirit of quantum shuffle algebras (\cite{Ro}), by
replacing the usual flip with a braiding. The resulting algebras,
called quantum quasi-shuffle algebras, are the generalization of
quantum shuffle algebras and provide Yang-Baxter algebras. Because
of the importance of quasi-shuffle algebras, it seems quite
reasonable to study quantum quasi-shuffle algebras for themselves
 as  new algebraic objects, not just as special quantum
$B_\infty$-algebras. This paper is the first step in this
direction. As a starting point, we expect that the quantum
quasi-shuffle algebra can inherit some good properties of the
classical one, or have some "q-analogues" of those in the
classical case. We first investigate when and how we can construct
the quantum quasi-shuffle product on the tensor space $T(V)$ over
a braided vector space $(V,\sigma)$. Universal properties always
play an important role in the study of algebras. So we provide the
universal property of quantum quasi-shuffle algebras in a suitable
category. We also study the commutativity of the algebra and
present a linear basis of $T(V)$ for a special kind of Yang-Baxter
algebras $(V,m,\sigma)$ by using Lyndon words.

This paper is organized as follows. In Section 2, we recall the
construction of quantum quasi-shuffle algebras and study the
 necessary and sufficient condition for the construction. In Section 3, we provide a
universal property of quantum quasi-shuffle algebras in the
category of connected twisted Yang-Baxter bialgebras and discuss
the commutativity of quantum quasi-shuffle algebras. In Section 4,
for a special kind of Yang-Baxter algebras $(V,m,\sigma)$, we
provide a linear basis of $T(V)$ by using the quantum
quasi-shuffle product and Lyndon words.

\section*{Notations}

In this paper, we denote by $K$ a ground field of characteristic
0. All the objects we discuss are defined over $K$.

The symmetric group of $n$ letters $\{1,2,\ldots,n\}$ is written
by $\mathfrak{S}_{n}$. An $(i,j)$-shuffle is an element $w\in
\mathfrak{S}_{i+j}$ such that $w (1) < \cdots <w (i)$ and $w (i+1)
< \cdots <w (i+j)$. We denote by $\mathfrak{S}_{i,j}$ the set of
all $(i,j)$-shuffles.

A braiding $\sigma$ on a vector space $V$ is an invertible linear
map in $\mathrm{End}(V\otimes V)$ satisfying the quantum
Yang-Baxter equation on $V^{\otimes 3}$: $$(\sigma\otimes
\mathrm{id}_{V})(\mathrm{id}_{V}\otimes \sigma)(\sigma\otimes
\mathrm{id}_{V})=(\mathrm{id}_{V}\otimes \sigma)(\sigma\otimes
\mathrm{id}_{V})(\mathrm{id}_{V}\otimes \sigma).$$

A braided vector space $(V,\sigma)$ is a vector space $V$ equipped
with a braiding $\sigma$. For any $n\in \mathbb{N}$ and $1\leq
i\leq n-1$, we denote by $\sigma_i$ the operator
$\mathrm{id}_V^{\otimes(i-1)}\otimes \sigma\otimes
\mathrm{id}_V^{\otimes(n-i-1)}\in \mathrm{End}(V^{\otimes n})$.
For any $w\in \mathfrak{S}_{n}$, we denote by $T_w$ the
corresponding lift of $w$ in the braid group $B_n$, defined as
follows: if $w=s_{i_1}\cdots s_{i_l}$ is any reduced expression of
$w$, where $s_{i}=(i,i+1)$, then $T_w=\sigma_{i_1}\cdots
\sigma_{i_l}$. This definition is well-defined (see, e.g., Theorem
4.12 in \cite{KT}). Sometimes we also use the notation
$T_w^\sigma$ to indicate the action of $\sigma$.

The usual flip switching two factors is denoted by $\tau$. For a
vector space $V$, we denote by $\otimes$ the tensor product within
$T(V)$, and by $\underline{\otimes}$ the one between $T(V)$ and
$T(V)$ respectively.

\section{Quantum quasi-shuffle algebras} We start by recalling some definitions. In the following, all algebras are
assumed to be associative and unital.
\begin{definition}[\cite{HH}]1. Let $A=(A,m)$ be an algebra with product $m$ and unit $1_A$, and $\sigma$ be a braiding on $A$. We call $(A,m,\sigma)$ a \emph{Yang-Baxter algebra} (YB algebra for short) if it satisfies the following conditions:
\[\left\{
\begin{array}{lll}
(\mathrm{id}_A\otimes m)\sigma_1\sigma_2&=&\sigma( m\otimes \mathrm{id}_A),\\
( m\otimes
\mathrm{id}_A)\sigma_2\sigma_1&=&\sigma(\mathrm{id}_A\otimes m),
\end{array} \right.
\]
and for any $a\in A$,
\[\left\{
\begin{array}{lll}
\sigma(1_A\otimes a)&=&a\otimes 1_A,\\
\sigma(a\otimes 1_A)&=&1_A\otimes a.
\end{array} \right.
\]

2. Let $C=(C,\bigtriangleup ,\varepsilon)$ be a coalgebra with
coproduct $\bigtriangleup$ and counit $\varepsilon$, and $\sigma$
be a braiding on $C$. We call $(C,\bigtriangleup,\sigma)$ a
\emph{Yang-Baxter coalgebra} (YB coalgebra for short) if it
satisfies the following conditions:
\[\left\{
\begin{array}{lll}
\sigma_1\sigma_2 (\bigtriangleup\otimes \mathrm{id}_C)&=& (\mathrm{id}_C\otimes \bigtriangleup)\sigma,\\
\sigma_2\sigma_1(\mathrm{id}_C\otimes
\bigtriangleup)&=&(\bigtriangleup\otimes \mathrm{id}_C)\sigma,
\end{array} \right.
\]
and \[\left\{
\begin{array}{lll}
(\mathrm{id}_C\otimes
\varepsilon )\sigma&=&\varepsilon\otimes \mathrm{id}_C,\\
(\varepsilon \otimes \mathrm{id}_C )\sigma&=&\mathrm{id}_C\otimes
\varepsilon .
\end{array} \right.
\]
\end{definition}

These definitions give a right way to generalize the usual algebra
(resp. coalgebra) structure on the tensor products of algebras
(resp. coalgebras) in the following sense.

\begin{proposition}[\cite{HH}, Proposition 4.2]1.
For a YB algebra $(A,m,\sigma)$ and any $i\in \mathbb{N}$,
$(A^{\otimes i},m_{\sigma,i}, T^\sigma_{\chi_{ii}})$ becomes a YB
algebra with product $m_{\sigma,i}=m^{\otimes i}\circ
T^\sigma_{w_i}$ and unit $1_ A^{\otimes i}$, where
$\chi_{ii},w_i\in \mathfrak{S}_{2i}$ are given by
\[\chi_{ii}=\left(\begin{array}{cccccccc}
1&2&\cdots&i&i+1&i+2&\cdots & 2i\\
i+1&i+2&\cdots&2i&1& 2 &\cdots & i
\end{array}\right),\] and \[w_{i}=\left(\begin{array}{ccccccccc}
1&2&3&\cdots&i&i+1&i+2&\cdots & 2i\\
1&3&5&\cdots&2i-1&2& 4 &\cdots & 2i
\end{array}\right).\]

2. For a YB coalgebra $(C, \bigtriangleup,\sigma)$, $(C^{\otimes
i},\bigtriangleup_{\sigma,i}, T^\sigma_{\chi_{ii}})$ becomes a YB
coalgebra with coproduct $\bigtriangleup_{\sigma,i}=
T^\sigma_{w_i^{-1}}\circ\bigtriangleup^{\otimes i}$ and counit
$\varepsilon^{\otimes i}:C^{\otimes i}\rightarrow K^{\otimes
i}\simeq K$.\end{proposition}

We call $m_\sigma=m_{\sigma,2}$ the \emph{twisted algebra
structure} on $A\otimes A$ and
$\bigtriangleup_{\sigma}=\bigtriangleup_{\sigma,2}$ the
\emph{twisted coalgebra structure} on $C\otimes C$.

Let $(V,\sigma)$ be a braided vector space. For any $i,j\geq 1$,
we denote\[\chi_{ij}=\left(\begin{array}{cccccccc}
1&2&\cdots&i&i+1&i+2&\cdots & i+j\\
j+1&j+2&\cdots&j+i&1& 2 &\cdots & j
\end{array}\right),\] and define $\beta:T(V)\underline{\otimes} T(V)\rightarrow T(V)\underline{\otimes} T(V)$ by requiring that $\beta_{ij}=T^\sigma_{\chi_{ij}}$ on $V^{\otimes i}\underline{\otimes} V^{\otimes
j}$. For convenience, we denote by $\beta_{0i}$ and $\beta_{i0}$
the usual flip map $\tau$.

Then $(T(V),m,\beta)$ is a YB algebra, where $m$ is the
concatenation product.

Another example of YB algebras is the quantum shuffle algebra (see
\cite{Ro}). For a braided vector space $(V,\sigma)$, one can
constuct an associative algebra structure on $T(V)$ by: for any
$x_1,\ldots, x_{i+j}\in V$, $$(x_1\otimes\cdots\otimes
x_i)\mbox{{\cyr sh}}_{\sigma}(x_{i+1}\otimes\cdots\otimes
x_{i+j})=\sum_{w\in \mathfrak{S}_{i,j}}T_w(x_1\otimes\cdots\otimes
x_{i+j}).$$ The space $T(V)$ equipped with $\mbox{{\cyr
sh}}_{\sigma}$ is called the \emph{quantum shuffle algebra} and
denoted by $T_\sigma(V)$. We have that $(T_\sigma(V), \beta)$ is a
YB algebra.

We define $\delta$ to be the deconcatenation on $T(V)$, i.e.,
$$\delta(v_1\otimes\cdots\otimes v_n)=\sum_{i=0}^n(v_1\otimes\cdots\otimes v_i)\underline{\otimes }(v_{i+1}\otimes\cdots\otimes
v_n).$$ We denote by $T^c(V)$ the coalgebra
$(T(V),\delta,\varepsilon)$ where $\varepsilon$ is the projection
from $T(V)$ to $K$. The coalgebra $T^c(V)$ is the cotensor
coalgebra (see \cite{N}) over the trivial Hopf algebra $K$. Here
$V$ is a Hopf bimodule with scalar multiplication and coactions
defined by $\delta_L(v)=1\otimes v$ and $\delta_R(v)=v\otimes 1$
for any $v\in V$. $(T^c(V),\beta)$ is a YB coalgebra.

Now we review the construction of the quantum quasi-shuffle
algebra which was given as a special example of quantum
$B_\infty$-algebras in \cite{JR}.

Let $(V, \sigma)$ be a braided vector space and  for any $ p,
q\geq 0$, $M_{pq}:V^{\otimes p}\otimes V^{\otimes q}\rightarrow V$
be a linear map  such that
\[\left\{
\begin{array}{lllll}
M_{00}&=&0, \\
M_{10}&=&\mathrm{id}_V=M_{01},\\
M_{11}&=&m,\\
M_{pq}&=&0,\ \mathrm{otherwise}.
\end{array} \right.
\]
We denote $$\Join_\sigma=\varepsilon\otimes
\varepsilon+\sum_{n\geq 1}M^{\otimes n}\circ
\bigtriangleup_\beta^{(n-1)}:T^c(V)\underline{\otimes}
T^c(V)\rightarrow T^c(V),$$ where $M=(M_{pq})_{p,q\geq 0}$,
$\bigtriangleup_\beta=(\mathrm{id}_{T^c(V)}\otimes
\beta\otimes\mathrm{id}_{T^c(V)})\circ (\delta\otimes \delta)$ and
$\bigtriangleup_\beta^{(n)}=(\bigtriangleup_\beta\otimes
\mathrm{id}_{T^c(V)}^{\otimes 2(n-1)} )\circ
\bigtriangleup_\beta^{(n-1)}$ inductively. It is easy to show by
induction that the summation with respect to $n$ in the above
formula is finite. Indeed, since $M^{\otimes (i+j+1)}\circ
\bigtriangleup_\beta^{(i+j)}=\Big((M^{\otimes (i+j)}\circ
\bigtriangleup_\beta^{(i+j-1)})\otimes
M\Big)\circ\bigtriangleup_\beta$ and the conditions for $M$, we
have that $M^{\otimes (i+j+1)}\circ
\bigtriangleup_\beta^{(i+j)}(x\underline{\otimes}y)=0$, for any
$x\in V^{\otimes i}$ and $y\in V^{\otimes j}$.

To illustrate the new map $\Join_\sigma$, we calculate a few
examples. For any $u,v,w\in V$, we have
\begin{eqnarray*}
  u\Join_\sigma v&=&(\varepsilon\otimes
\varepsilon+M\circ \bigtriangleup_\beta^{(0)}+M^{\otimes 2}\circ
\bigtriangleup_\beta^{(1)})(u\underline{\otimes} v)
\\[3pt]
&=&M_{11}(u\underline{\otimes} v)\\[3pt]
&&+M^{\otimes 2}\Big(1\underline{\otimes} \beta_{10} (u\underline{\otimes }1)\underline{\otimes }v+1\underline{\otimes}\beta_{11}(u\underline{\otimes} v)\underline{\otimes} 1\\[3pt]
&&+u\underline{\otimes }\beta_{00}(1\underline{\otimes }1)\underline{\otimes} v+ u\underline{\otimes }\beta_{01}(1\underline{\otimes } v)\underline{\otimes } 1\Big)\\[3pt]
&=&M_{11}(u\underline{\otimes} v)+(M_{01}\otimes M_{10})(1\underline{\otimes }\sigma(u\underline{\otimes } v)\underline{\otimes } 1)\\[3pt]
&&+(M_{10}\otimes M_{01})((u\underline{\otimes }1)\underline{\otimes }(1\underline{\otimes} v))\\[3pt]
&=&M_{11}(u\underline{\otimes} v)+u\underline{\otimes} v+\sigma(u\underline{\otimes } v)\\[3pt]
&=&M_{11}(u{\otimes} v)+u\mbox{\cyr sh}_\sigma v,
\end{eqnarray*}
\begin{eqnarray*}
\lefteqn{(u\otimes v)\Join_\sigma w}\\
&=&(\varepsilon\otimes \varepsilon+M\circ\bigtriangleup_\beta^{(0)}+M^{\otimes 2}\circ \bigtriangleup_\beta^{(1)}+M^{\otimes 3}\circ \bigtriangleup_\beta^{(2)})((u\otimes v)\underline{\otimes} w) \\[3pt]
&=&M^{\otimes 2}\Big(1\underline{\otimes } \beta_{20}((u\otimes v)\underline{\otimes} 1)\underline{\otimes } w+u\underline{\otimes } \beta_{10}(v\underline{\otimes } 1)\underline{\otimes } w\\[3pt]
&&+(u\otimes v)\underline{\otimes } \beta_{00}(1\underline{\otimes } 1)\underline{\otimes } w+ 1\underline{\otimes } \beta_{21}((u\otimes v)\underline{\otimes} w)\underline{\otimes } 1\\[3pt]
&&+ u\underline{\otimes } \beta_{11}( v\underline{\otimes }
w)\underline{\otimes } 1+(u\otimes v)\underline{\otimes }
\beta_{01}(1\underline{\otimes } w)\underline{\otimes } 1\Big)\\[3pt]
&&+M^{\otimes 3}\Big(\bigtriangleup_\beta(u\underline{\otimes}
1)\underline{\otimes}(v\underline{\otimes}
w)+\bigtriangleup_\beta((u\otimes v)\underline{\otimes}
1)\underline{\otimes}(1\underline{\otimes}
w)\\[3pt]
&&+(\bigtriangleup_\beta\otimes id_{T^c(V)^{\otimes
2}})(1\underline{\otimes } \beta_{21}((u\otimes
v)\underline{\otimes} w)\underline{\otimes }
1)\\[3pt]
&&+(\bigtriangleup_\beta\otimes id_{T^c(V)^{\otimes
2}})(u\underline{\otimes }
\beta_{11}( v\underline{\otimes } w)\underline{\otimes } 1)\Big)\\[3pt]
&=&u\otimes M_{11}(v\underline{\otimes}w)+(M_{11}\otimes
M_{10})(u\underline{\otimes } \sigma( v\underline{\otimes }
w)\underline{\otimes }
1)\\
&&+M^{\otimes 3}\Big(u\underline{\otimes}
\beta_{10}(v\underline{\otimes } 1)\underline{\otimes }
1\underline{\otimes}(1\underline{\otimes}
w)+(\bigtriangleup_\beta\otimes id_{T^c(V)^{\otimes
2}})(u\underline{\otimes}
\beta_{11}( v\underline{\otimes} w)\underline{\otimes} 1)\Big)\\
&=&u\otimes M_{11}(v\underline{\otimes}w)+(M_{11}\otimes
\mathrm{id}_V)(u\otimes \sigma( v\otimes w))\\
&&+u\otimes v\otimes w+\sigma_2(u\otimes v\otimes
w)+\sigma_1\sigma_2(u\otimes v\otimes w)\\
&=&u\otimes M_{11}(v{\otimes}w)+(M_{11}\otimes
\mathrm{id}_V)(u\otimes \sigma( v\otimes w))+(u\otimes v
)\mbox{\cyr sh}_\sigma w,\end{eqnarray*} and
\begin{displaymath}
u\Join_\sigma(v\otimes w)= M_{11}(u{\otimes}v)\otimes w
+(\mathrm{id}_V\otimes M_{11})(\sigma( u\otimes v)\otimes
w)+u\mbox{\cyr sh}_\sigma (v\otimes w).\end{displaymath}

We denote by $\Join_{\sigma (i,j)}$ the restriction of
$\Join_\sigma$ on $V^{\otimes i}\underline{\otimes}V^{\otimes j}$.
Then we have the following inductive formula.

\begin{proposition}For $i,j>1$ and any $u_1,\ldots,
u_i,v_1,\ldots, v_j\in V$, we have
\begin{eqnarray}
\lefteqn{(u_1\otimes\cdots\otimes u_i)\Join_\sigma (v_1\otimes\cdots\otimes v_j)}\nonumber\\
&=&\Big((u_1\otimes\cdots\otimes u_i)\Join_\sigma (v_1\otimes\cdots\otimes v_{j-1})\Big)\otimes v_j\nonumber\\
&&+(\Join_{\sigma (i-1,j)}\otimes \mathrm{id}_V)\sigma_{i+j-1}\cdots\sigma_i(u_1\otimes\cdots\otimes u_i\otimes v_1\otimes\cdots\otimes v_j)\nonumber\\
&&+(\Join_{\sigma (i-1,j-1)}\otimes
m)\sigma_{i+j-2}\cdots\sigma_i(u_1\otimes\cdots\otimes u_i\otimes
v_1\otimes\cdots\otimes v_j).
\end{eqnarray}\end{proposition}
\begin{proof}By the definition of $\Join_\sigma$, we have
\begin{eqnarray*}
\lefteqn{(u_1\otimes\cdots\otimes u_i)\Join_\sigma (v_1\otimes\cdots\otimes v_j)}\\
&=&(\varepsilon\otimes \varepsilon+\sum_{n\geq 1}M^{\otimes
n}\circ \bigtriangleup_\beta^{(n-1)})\Big((u_1\otimes\cdots\otimes
u_i)\underline{\otimes}(v_1\otimes\cdots\otimes v_j)\Big)\\
&=&\sum_{n=2}^{i+j}\Big((M^{\otimes (n-1)}\circ
\bigtriangleup_\beta^{(n-2)})\otimes
M\Big)\circ\bigtriangleup_\beta\Big((u_1\otimes\cdots\otimes
u_i)\underline{\otimes}(v_1\otimes\cdots\otimes v_j)\Big)\\
&=&\sum_{n=2}^{i+j}(M^{\otimes (n-1)}\circ
\bigtriangleup_\beta^{(n-2)})\Big((u_1\otimes\cdots\otimes
u_i)\underline{\otimes}(v_1\otimes\cdots\otimes
v_{j-1})\Big)\otimes
M_{01}(1\underline{\otimes}v_j)\\
&&+\sum_{n=2}^{i+j}\Big((M^{\otimes (n-1)}\circ
\bigtriangleup_\beta^{(n-2)})\otimes M_{10}\Big)\\
&&\ \ \ \ \ \ \ \ \ \ \circ(\mathrm{id}_V^{\otimes
i-1}\otimes\beta_{1j}\otimes\mathrm{id}_K)\Big((u_1\otimes\cdots\otimes
u_{i-1})\underline{\otimes}u_i\underline{\otimes}(v_1\otimes\cdots\otimes
v_{j})\underline{\otimes}1\Big)\\
&&+\sum_{n=2}^{i+j}\Big((M^{\otimes (n-1)}\circ
\bigtriangleup_\beta^{(n-2)})\otimes M_{11}\Big)\\
&&\ \ \ \ \ \ \ \ \ \ \circ(\mathrm{id}_V^{\otimes
i-1}\otimes\beta_{1,j-1}\otimes\mathrm{id}_V)\Big((u_1\otimes\cdots\otimes
u_{i-1})\underline{\otimes}u_i\underline{\otimes}(v_1\otimes\cdots\otimes
v_{j-1})\underline{\otimes}v_j\Big)\\
&=&\Big((u_1\otimes\cdots\otimes u_i)\Join_\sigma (v_1\otimes\cdots\otimes v_{j-1})\Big)\otimes v_j\\
&&+(\Join_{\sigma (i-1,j)}\otimes \mathrm{id}_V)\sigma_{i+j-1}\cdots\sigma_i(u_1\otimes\cdots\otimes u_i\otimes v_1\otimes\cdots\otimes v_j)\nonumber\\
&&+(\Join_{\sigma (i-1,j-1)}\otimes
m)\sigma_{i+j-2}\cdots\sigma_i(u_1\otimes\cdots\otimes u_i\otimes
v_1\otimes\cdots\otimes v_j).
\end{eqnarray*}\end{proof}
 It is
easy to see that $1\Join_\sigma x=x \Join_\sigma 1=x$ for any
$x\in T^c(V)$, where $1$ is the unit of $K$. But there is no
evidence that $(T^c(V),\Join_\sigma)$ is an associative algebra.
Actually, $\Join_\sigma$ is not associative in general. If the map
$M_{11}=m$ is an associative product and is compatible with the
braiding $\sigma$ in some sense, it comes true.
\begin{theorem}Under the above notations, $(T^c(V),\Join_\sigma,\beta)$ is
a YB algebra if and only if $(V,M_{11},\sigma)$ is a YB
algebra.\end{theorem}

\begin{proof}Let $(V,M_{11},\sigma)$ be a YB algebra. One can find a detailed proof in \cite{JR} for that $(T^c(V),\Join_\sigma,\beta)$ is
a YB algebra in a more general setting. Because of the simplicity
of the $M$ here, we can provide another proof by using the
relation (1). Since $(T^c(V),\beta)$ is a YB coalgebra, we know
that $(T^c(V)^{\otimes
2},\bigtriangleup_\beta,T^\beta_{\chi_{22}})$ is a YB coalgebra by
Proposition 2. By using the compatibility conditions for $M_{11}$
and $\sigma$ and those for $\bigtriangleup_\beta$ and $\beta$, it
is easy to prove that
\[\left\{
\begin{array}{lll}
\beta(\Join_\sigma \otimes \mathrm{id}_{T^c(V)})&=&(
\mathrm{id}_{T^c(V)}\otimes \Join_\sigma )\beta_1\beta_2 ,\\[3pt]
\beta(\mathrm{id}_{T^c(V)}\otimes
\Join_\sigma)&=&(\Join_\sigma\otimes \mathrm{id}_{T^c(V)}
)\beta_2\beta_1.
\end{array} \right.
\]
Now we show that $\Join_\sigma$ is associative, i.e., for any
$x\in V^{\otimes i}$, $y\in V^{\otimes j}$, $z\in V^{\otimes k}$,
we have $x\Join_\sigma(y\Join_\sigma z)=(x\Join_\sigma
y)\Join_\sigma z$. If $j=0$ or $k=0$, there is nothing to prove.
So we assume that $j,k\geq 1$, $y=y^\prime\otimes u$ and
$z=z^\prime\otimes v$, where $u, v\in V$ and $y^\prime,
z^\prime\in T^c(V)$. By the inductive formula (1) and the above
compatibility conditions for $\Join_\sigma$ and $\beta$, it is
easy to prove the statement by using induction on $i+j+k$.

Conversely, if $\Join_\sigma$ is associative, then for any
$u,v,w\in V$, we have
\begin{eqnarray*}
(u\Join_\sigma v)\Join_\sigma w&=&(M_{11}(u{\otimes} v)+u\mbox{\cyr sh}_\sigma v)\Join_\sigma w\\
&=&M_{11}(M_{11}(u{\otimes} v){\otimes} w)+M_{11}(u{\otimes} v)\mbox{\cyr sh}_\sigma w\\
&&+(\mathrm{id}_V\otimes M_{11})(u\mbox{\cyr sh}_\sigma
v{\otimes}w)+(M_{11}\otimes \mathrm{id})(\mathrm{id}_V\otimes
\sigma)(u\mbox{\cyr sh}_\sigma
v{\otimes} w)\\
&&+u\mbox{\cyr sh}_\sigma v
\mbox{\cyr sh}_\sigma w\\
&=&[M_{11}(M_{11}\otimes \mathrm{id}_V)+(
\mathrm{id}_V+\sigma)(M_{11}\otimes \mathrm{id}_V)\\
&&+(\mathrm{id}_V\otimes M_{11})((\mathrm{id}_V+\sigma)\otimes \mathrm{id}_V)\\
&&+(M_{11}\otimes
\mathrm{id}_V)(\sigma_2+\sigma_2\sigma_1)](u{\otimes} v{\otimes} w)\\
&&+u\mbox{\cyr sh}_\sigma v \mbox{\cyr sh}_\sigma w,
\end{eqnarray*}
and
\begin{eqnarray*}
u\Join_\sigma(v\Join_\sigma w)&=&u\Join_\sigma(M_{11}(v{\otimes} w)+v\mbox{\cyr sh}_\sigma w)\\
&=&M_{11}(u{\otimes} M_{11}(v{\otimes} w))+u\mbox{\cyr sh}_\sigma M_{11}(v{\otimes} w)\\
&&+( M_{11}\otimes\mathrm{id}_V)(u {\otimes} v\mbox{\cyr
sh}_\sigma w)+(\mathrm{id}_V\otimes
M_{11})(\sigma\otimes\mathrm{id}_V )(u{\otimes}
v \mbox{\cyr sh}_\sigma w))\\
&&+u\mbox{\cyr sh}_\sigma v
\mbox{\cyr sh}_\sigma w\\
&=&[M_{11}(\mathrm{id}_V\otimes M_{11})+(
\mathrm{id}_V+\sigma)(\mathrm{id}_V\otimes M_{11})\\
&&+(M_{11}\otimes \mathrm{id}_V)(\mathrm{id}_V\otimes (\mathrm{id}_V+\sigma))\\
&&+(\mathrm{id}_V
\otimes M_{11})(\sigma_1+\sigma_1\sigma_2)](u{\otimes} v{\otimes} w)\\
&&+u\mbox{\cyr sh}_\sigma v \mbox{\cyr sh}_\sigma w.
\end{eqnarray*}
Therefore we have that $(u\Join_\sigma v)\Join_\sigma
w=u\Join_\sigma (v\Join_\sigma w)$ if and only if
\begin{eqnarray*}
&&M_{11}(M_{11}\otimes \mathrm{id}_V)+M_{11}\otimes \mathrm{id}_V
+\sigma(M_{11}\otimes \mathrm{id}_V)\\
&&+\mathrm{id}_V\otimes M_{11}+(\mathrm{id}_V\otimes
M_{11})\sigma_1 +(M_{11}\otimes
\mathrm{id}_V)\sigma_2+(M_{11}\otimes
\mathrm{id}_V)\sigma_2\sigma_1\\
&=&M_{11}(\mathrm{id}_V\otimes M_{11})+\mathrm{id}_V\otimes
M_{11}+\sigma(\mathrm{id}_V\otimes M_{11})\\
&&+M_{11}\otimes \mathrm{id}_V+(M_{11}\otimes
\mathrm{id}_V)\sigma_2+(\mathrm{id}_V \otimes
M_{11})\sigma_1+(\mathrm{id}_V \otimes M_{11})\sigma_1\sigma_2,
\end{eqnarray*}
i.e.,
\begin{eqnarray*}
\lefteqn{M_{11}(M_{11}\otimes \mathrm{id}_V) +\sigma(M_{11}\otimes
\mathrm{id}_V)+(M_{11}\otimes
\mathrm{id}_V)\sigma_2\sigma_1}\\
&=&M_{11}(\mathrm{id}_V\otimes M_{11})+\sigma(\mathrm{id}_V\otimes
M_{11})+(\mathrm{id}_V \otimes M_{11})\sigma_1\sigma_2.
\end{eqnarray*}
By comparing the degrees of the resulting tensor vectors, we must
have $M_{11}(M_{11}\otimes
\mathrm{id}_V)=M_{11}(\mathrm{id}_V\otimes M_{11})$.

On $V{\otimes}V{\otimes}V$, the condition
$(\mathrm{id}_V\otimes\Join_\sigma)\sigma_1\sigma_2=\sigma(\Join_\sigma\otimes
\mathrm{id}_V)$ implies that $(\mathrm{id}_V\otimes
M_{11})\sigma_1\sigma_2+(\mathrm{id}_V\otimes\mbox{\cyr
sh}_\sigma)\sigma_1\sigma_2=\sigma(M_{11}\otimes
\mathrm{id}_V)+\sigma(\mbox{\cyr sh}_\sigma\otimes
\mathrm{id}_V)$. By comparing the degrees, we get
$(\mathrm{id}_V\otimes
M_{11})\sigma_1\sigma_2=\sigma(M_{11}\otimes \mathrm{id}_V)$.
Similarly, we have $(M_{11}\otimes
\mathrm{id}_V)\sigma_2\sigma_1=\sigma(\mathrm{id}_V\otimes
M_{11})$.
\end{proof}
The above resulting algebra $(T^c(V),\Join_\sigma)$ is called the
\emph{quantum quasi-shuffle algebra} over $(V,M_{11},\sigma)$.

\begin{remark}For an algebra $(A,m)$, the quantum quasi-shuffle algebra over the trivial YB algebra $(A,m,\tau)$ is just the classical quasi-shuffle algebra over the algebra $A$.\end{remark}

\section{Universal property and commutativity}
Let $(C,\bigtriangleup, \varepsilon)$ be a coalgebra with a
preferred group-like element $1_C\in C$. We denote
$\overline{\bigtriangleup}(x)=\bigtriangleup(x)-x\otimes
1_C-1_C\otimes x$ for any $x\in C$. The map
$\overline{\bigtriangleup}$ is also coassociative and called the
\emph{reduced coproduct}. We also denote
$\overline{C}=\mathrm{Ker}\varepsilon$. Then $C=K1_C\oplus
\overline{C}$ since $x-\varepsilon(x)1_C\in \overline{C}$ for any
$x\in C$.
\begin{definition}[\cite{Q}]The coalgebra $(C,\bigtriangleup)$ is said to be \emph{connected} if $C=\cup_{r\geq 0}F_r C$,
where\begin{eqnarray*}F_0C&=&K1_C,\\
F_rC&=&\{x\in C|\overline{\bigtriangleup}(x)\in F_{r-1}C\otimes
F_{r-1}C \},\ \ \mathrm{for\ r\geq
1}.\end{eqnarray*}\end{definition}

There is a well-known universal property for $T^c(V)$ (see, e.g.,
\cite{LoR}):
\begin{proposition}Given a connected coalgebra $(C,\bigtriangleup,\varepsilon)$ and a linear map $\phi:C\rightarrow V$ such that $\phi(1_C)=0$. Then there is a unique coalgebra morphism $\overline{\phi}:C\rightarrow T^c(V)$ which extends $\phi$, i.e., $Pr_V\circ \overline{\phi}=\phi$, where $Pr_V:T^c(V)\rightarrow V $ is the projection onto $V$. Explicitly, $\overline{\phi}=\varepsilon+\sum_{n\geq 1}\phi^{\otimes n}\circ\overline{\bigtriangleup}^{(n-1)}$, where $\overline{\bigtriangleup}^{(n)}=(\overline{\bigtriangleup}^{(n-1)}\otimes\mathrm{id}_C)\circ \overline{\bigtriangleup}$ inductively.\end{proposition}

\begin{corollary}Let $C$ be a connected coalgebra. If $\Phi,\Psi:C\rightarrow T^c(V)$ are coalgebra maps such that $Pr_V\circ \Phi=Pr_V\circ \Psi$ and $Pr_V\circ\Phi(1_C)=0=Pr_V\circ\Psi(1_C)$, then $\Phi=\Psi$. \end{corollary}

We will use the above properties to provide a universal property
of the quantum quasi-shuffle algebra $(T^c(V),\Join_\sigma)$ in
some category. First we describe the category in which we will
work.

\begin{definition}A quadruple $(H,\cdot,\bigtriangleup,\sigma)$ is called a \emph{twisted Yang-Baxter (YB for short) bialgebra} if

1. $(H,\cdot, \sigma)$ is a YB algebra,

2. $(H,\bigtriangleup, \sigma)$ is a YB coalgebra,

3. $\cdot: H\otimes H\rightarrow H$ is a coalgebra map, where
$H\otimes H$ is equipped with the twisted coalgebra structure. Or
equivalently, $\bigtriangleup: H\rightarrow H\otimes H$ is an
algebra map, where $H\otimes H$ is equipped with the twisted
algebra structure.
\end{definition}

From the condition 3 above, we have that
$\bigtriangleup(1_H)=1_H\otimes 1_H$.

\noindent\textbf{Examples}. 1. Let $(V,\sigma)$ be a braided
vector space. Then the quantum shuffle algebra $(T_\sigma(V),
\beta)$ equipped with the deconcatenation coproduct $\delta$ is a
twisted YB bialgebra (see \cite{Ro}).

2. Let $(V,m,\sigma)$ be a YB algebra. Then the quantum
quasi-shuffle algebra $(T^c(V),\Join_\sigma,\beta)$ is a twisted
YB bialgebra with the deconcatenation coproduct $\delta$ (see
\cite{JR}).

We denote by $CB_{\mathcal{YB}}$ the category of connected twisted
YB bialgebras. It consists of the following data:

1. the objects of $CB_{\mathcal{YB}}$ are the twisted YB
bialgebras $(H,\cdot,\bigtriangleup,\sigma)$ such that both $H$
and $H\otimes H$ are connected, where the preferred group-like
elements are $1_H$ and $1_H\otimes 1_H$ respectively, and
$H\otimes H$ is equipped with the twisted coalgebra structure;

2. a morphism $f$ from object $(H_1,\sigma_1)$ to object
$(H_2,\sigma_2)$ is both an algebra map and a coalgebra map and
satisfies that $(f\otimes f)\sigma_1=\sigma_2(f\otimes f)$.

It is easy to see that both $(T(V),\mbox{{\cyr
sh}}_{\sigma},\delta,\beta)$ and $(T^c(V),\Join_\sigma,\delta,
\beta)$ are in $CB_{\mathcal{YB}}$.

\begin{lemma}Let $(V_1, \sigma_1)$ and $(V_2, \sigma_2)$ be two braided vector spaces and $f: V_1\rightarrow V_2$ be a morphism of braided vector spaces, i.e. a linear map such that $\sigma_2(f\otimes f)=(f\otimes f)\sigma_1$. Then for any $i,j\geq 1$, $T_{\chi_{ij}}^{\sigma_2}(f^{\otimes i}\otimes f^{\otimes j})=(f^{\otimes j}\otimes f^{\otimes i})T_{\chi_{ij}}^{\sigma_1}$.\end{lemma}
\begin{proof}We use induction on $i+j$.

When $i=j=1$, it is trivial.

For $i+j\geq 3$, we have
\begin{eqnarray*}
T_{\chi_{ij}}^{\sigma_2}(f^{\otimes i}\otimes f^{\otimes j})&=&(T_{\chi_{i-1,j}}^{\sigma_2}\otimes \mathrm{id}_{V_2})(\mathrm{id}_{V_2}^{\otimes i-1}\otimes T_{\chi_{1,j}}^{\sigma_2})(f^{\otimes i}\otimes f^{\otimes j})\\
&=&(T_{\chi_{i-1,j}}^{\sigma_2}\otimes \mathrm{id}_{V_2})(f^{\otimes i-1}\otimes T_{\chi_{1,j}}^{\sigma_2}(f\otimes f^{\otimes j}))\\
&=&(T_{\chi_{i-1,j}}^{\sigma_2}\otimes \mathrm{id}_{V_2})(f^{\otimes i-1}\otimes f^{\otimes j}\otimes f)(\mathrm{id}_{V_1}^{\otimes i-1}\otimes T_{\chi_{1,j}}^{\sigma_1})\\
&=&(f^{\otimes j}\otimes f^{\otimes i})(T_{\chi_{i-1,j}}^{\sigma_1}\otimes \mathrm{id}_{V_1})(\mathrm{id}_{V_1}^{\otimes i-1}\otimes T_{\chi_{1,j}}^{\sigma_1})\\
&=&(f^{\otimes j}\otimes f^{\otimes i})T_{\chi_{ij}}^{\sigma_1}.
\end{eqnarray*}
\end{proof}

\begin{lemma}Let $(C,\bigtriangleup,\sigma)$ be a YB coalgebra and $1_C$ be a group-like element of $C$. If $\sigma(1_C\otimes x)=x\otimes 1_C$ and $\sigma(x\otimes 1_C)=1_C\otimes x$ for any $x\in C$, then we have \[\left\{
\begin{array}{lll}
(\mathrm{id}_C\otimes \overline{\bigtriangleup})\sigma&=&\sigma_1\sigma_2( \overline{\bigtriangleup}\otimes\mathrm{id}_C) ,\\
(\overline{\bigtriangleup}\otimes\mathrm{id}_C)\sigma&=&\sigma_2\sigma_1(\mathrm{id}_C\otimes
\overline{\bigtriangleup}).
\end{array} \right.
\]\end{lemma}
\begin{proof}It follows direct computations.\end{proof}

Let $(V,m,\sigma)$ be a YB algebra. We have the following
universal property in $CB_{\mathcal{YB}}$.

\begin{proposition}For any $(H,\cdot,\bigtriangleup_H,\alpha)\in CB_{\mathcal{YB}}$ and a linear map $f:H\rightarrow V$ such that $m\circ (f\otimes f)=f\circ \cdot$ on $\overline{H}\otimes \overline{H}$, $f(1_H)=0$ and $(f\otimes f)\alpha=\sigma(f\otimes f)$, there exists a unique morphism $\overline{f}:H\rightarrow (T^c(V),\Join_\sigma,\delta,\beta)$ which extends $f$.\end{proposition}

\begin{proof}Observe that the condition on $f$ means that: $\forall x, y \in H, f(xy) = f(x)f(y) + \varepsilon (x)f(y) + \varepsilon (y)f(x)$. Since $f(1_H)=0$ and $H$ is connected, there is a unique coalgebra map $\overline{f}: H\rightarrow T^c(V)$ which extends $f$. More precisely, $\overline{f}=\varepsilon_H+\sum_{n\geq 1}f^{\otimes n}\circ \overline{\bigtriangleup_H}^{(n-1)}$.

We first prove that $\beta(\overline{f}\otimes
\overline{f})=(\overline{f}\otimes \overline{f})\alpha$. We only
need to verify it on $\overline{H}\otimes \overline{H}$. We have
\begin{eqnarray*}
\beta(\overline{f}\otimes
\overline{f})&=&\beta(\sum_{i,j\geq 1}(f^{\otimes i}\otimes f^{\otimes j})(\overline{\bigtriangleup_H}^{(i-1)}\otimes \overline{\bigtriangleup_H}^{(j-1)}))\\
&=&\sum_{i,j\geq 1}T_{\chi_{ij}}^{\sigma}(f^{\otimes i}\otimes f^{\otimes j})(\overline{\bigtriangleup_H}^{(i-1)}\otimes \overline{\bigtriangleup_H}^{(j-1)})\\
&=&\sum_{i,j\geq 1}(f^{\otimes j}\otimes f^{\otimes i})T_{\chi_{ij}}^{\alpha}(\overline{\bigtriangleup_H}^{(i-1)}\otimes \overline{\bigtriangleup_H}^{(j-1)})\\
&=&\sum_{i,j\geq 1}(f^{\otimes j}\otimes f^{\otimes i})(\overline{\bigtriangleup_H}^{(j-1)}\otimes \overline{\bigtriangleup_H}^{(i-1)})\alpha\\
&=&(\overline{f}\otimes \overline{f})\alpha,
\end{eqnarray*}
where the third and the forth equalities follow from Lemma 10 and
Lemma 11 respectively.

The next step is to prove that $\overline{f}$ is an algebra map.
We define two maps:
\begin{eqnarray*}
F_1: H\otimes H &\rightarrow& T^c(V),\\
 h\otimes g &\mapsto & \overline{f}(h)\Join_\sigma
 \overline{f}(g),
\end{eqnarray*}
and
\begin{eqnarray*}
F_2: H\otimes H &\rightarrow& T^c(V),\\
 h\otimes g &\mapsto & \overline{f}(hg).
\end{eqnarray*}
We claim that both $F_1$ and $F_2$ are coalgebra maps, where $
H\otimes H$ is equipped with the twisted coalgebra structure.

Indeed,
\begin{eqnarray*}
\delta \circ F_1&=&\delta\circ\Join_\sigma (\overline{f}\otimes \overline{f})\\
&=&(\Join_\sigma \otimes \Join_\sigma)\bigtriangleup_\beta (\overline{f}\otimes \overline{f})\\
&=&(\Join_\sigma \otimes \Join_\sigma)(\mathrm{id}_{T^c(V)}\otimes \beta \otimes \mathrm{id}_{T^c(V)})(\delta\otimes \delta)(\overline{f}\otimes \overline{f})\\
&=&(\Join_\sigma \otimes \Join_\sigma)(\mathrm{id}_{T^c(V)}\otimes \beta \otimes \mathrm{id}_{T^c(V)})(\delta\circ\overline{f}\otimes  \delta\circ \overline{f})\\
&=&(\Join_\sigma \otimes \Join_\sigma)(\mathrm{id}_{T^c(V)}\otimes \beta \otimes \mathrm{id}_{T^c(V)})(\overline{f}\otimes  \overline{f}\otimes\overline{f}\otimes \overline{f})(\bigtriangleup_H\otimes \bigtriangleup_H)\\
&=&(\Join_\sigma \otimes \Join_\sigma)(\overline{f}\otimes \beta(  \overline{f}\otimes\overline{f})\otimes \overline{f})(\bigtriangleup_H\otimes \bigtriangleup_H)\\
&=&(F_1\otimes F_1)(\mathrm{id}_H\otimes \alpha \otimes
\mathrm{id}_H)(\bigtriangleup_H\otimes \bigtriangleup_H)\\
&=&(F_1\otimes F_1)\bigtriangleup_{\alpha}.
\end{eqnarray*}
And
\begin{eqnarray*}
\delta \circ F_2&=&\delta\circ\overline{f}\circ\cdot\\
&=&( \overline{f}\otimes\overline{f})\circ\bigtriangleup_H\circ \cdot\\
&=&( \overline{f}\otimes\overline{f})(\cdot\otimes
\cdot)(\mathrm{id}_H\otimes \alpha \otimes
\mathrm{id}_H)(\bigtriangleup_H\otimes \bigtriangleup_H)\\
&=&(F_2\otimes F_2)\bigtriangleup_{ \alpha}.
\end{eqnarray*}
For any $h,g\in \overline{H}$, we have
\begin{eqnarray*}
Pr_V \circ F_1(h\otimes g)&=&Pr_V(\overline{f}(h)\Join_\sigma
 \overline{f}(g))\\
&=&Pr_V\big(\sum_{n\geq 1}M^{\otimes n}\bigtriangleup_\beta
^{(n-1)}(\overline{f}(h)\otimes
 \overline{f}(g))\big)\\
&=&M(\overline{f}(h)\otimes
 \overline{f}(g))\\
&=&\sum_{i,j\geq 1}M_{ij}((f^{\otimes i}\otimes f^{\otimes j})(\overline{\bigtriangleup_H}^{(i-1)}(h)\otimes \overline{\bigtriangleup_H}^{(j-1)}(g))\\
&=&M_{11}(f\otimes f)(h\otimes g)\\
&=&f\circ \cdot (h\otimes g)\\
&=&Pr_V\circ F_2(h\otimes g).
\end{eqnarray*}

Now, for any $h, g \in H$, write $h = \overline{h} +
\varepsilon(h)1, g = \overline{g} + \varepsilon(g)1$. We have:
$h\otimes g =\overline{h}\otimes \overline{g} +
\varepsilon(h)1\otimes \overline{g}+ \varepsilon(g)
\overline{h}\otimes 1 + \varepsilon(h) \varepsilon(g)1\otimes1$,
so $Pr_V \circ F_1(h\otimes g) = Pr_V \circ
F_1(\overline{h}\otimes\overline{ g}) +
\varepsilon(h)f(\overline{g}) + \varepsilon(g)f(\overline{h})$.
Also, $hg= \overline{h}\overline{g} + \varepsilon(h)\overline{g} +
\varepsilon(g)\overline{h} + \varepsilon(h)\varepsilon(g)1$, so
$Pr_V\circ F_2(h\otimes g)= Pr_V\circ F_2(\overline{h}\otimes
\overline{g}) + \varepsilon(h)f(\overline{g}) +
\varepsilon(g)f(\overline{h})$, and we have again equality. Since
$H\otimes H$ is connected with the twisted coalgebra structure,
$F_1=F_2$ follows from the Corollary 8.
\end{proof}

Now we begin to discuss the commutativity of quantum quasi-shuffle
algebras. In the classical case, if $A$ is a commutative algebra,
then the quasi-shuffle algebra built on $A$ is also commutative
(see, e.g., \cite{Hof, NR}). But for quantum quasi-shuffle
algebras, because of the complexity of the braiding, it is not
reasonable and  in fact not possible to require this even thought
the YB algebra is commutative in the usual sense. It is much more
suitable to discuss the commutativity in the sense of braided
category. This demands extra crucial conditions for the braiding
and the multiplication.

\begin{definition}A YB algebra $(A,m,\sigma)$ is called \emph{twisted
commutative} if $m\circ \sigma=m$.\end{definition}

\noindent\textbf{Examples}. 1. Let $(A,m)$ be an algebra. Then the
trivial YB algebra structure $(A,m,\tau)$ is twisted commutative
if and only if $A$ is commutative.

2. Let $V$ be a vector space over $\mathbb{C}$ with basis
$\{e_1,\ldots, e_N\}$. Take a nonzero scalar $q\in \mathbb{C}$. We
define a braiding $\sigma$ on $V$ by
\[\sigma(e_{i}\otimes e_{j})=\left\{
\begin{array}{lll}
e_{i}\otimes e_{i},&& i=j,\\
q^{-1}e_{j}\otimes e_{i},&& i<j,\\
q^{-1}e_{j}\otimes e_{i}+(1-q^{-2})e_{i}\otimes e_{j},&&i>j.
\end{array} \right.
\]
Then $\sigma$ satisfies the Iwahori's quadratic equation
$(\sigma-\mathrm{id}_{V\otimes
V})(\sigma+q^{-2}\mathrm{id}_{V\otimes V})=0.$ In fact, this
$\sigma$ is given by the $R$-matrix in the fundamental
representation of $U_q\mathfrak{sl}_N$. We denote
$\bigwedge_\sigma(V)=T(V)/I$, where $I$ is the ideal of $T(V)$
generated by $\mathrm{Ker}(\mathrm{id}_{V^{\otimes 2}}-\sigma)$.
By an easy computation, we get that
$\mathrm{Ker}(\mathrm{id}_{V\otimes
V}-\sigma)=\mathrm{Span}_{\mathbb{C}}\{ e_{i}\otimes
e_{i},q^{-1}e_{i}\otimes e_{j}+e_{j}\otimes e_{i}(i<j)\}$. We
denote by $e_{i_{1}}\wedge \cdots \wedge e_{i_{s}}$ the image of
$e_{i_{1}}\otimes \cdots \otimes e_{i_{s}}$ in
$\bigwedge_\sigma(V)$. So $\bigwedge_\sigma(V)$ is an algebra
generated by $(e_i)$ and the relations $e_i^2=0$ and $e_j\wedge
e_i=-q^{-1}e_i\wedge e_j$ if $i<j$. And the set $\{e_{i_{1}}\wedge
\cdots \wedge e_{i_{p}}|1\leq i_{1} < \cdots < i_{p}\leq N,1\leq
p\leq N\}$ forms a linear basis of $\bigwedge_\sigma(V)$. The
algebra $\bigwedge_\sigma(V)$ is called the \emph{quantum exterior
algebra} over $V$.

We denote the increasing set $(i_{1}, \ldots ,i_{s})$ by
$\underline{i}$ and so on. For $1\leq i_{1}<\cdots <i_{s}\leq N$
and $1\leq j_{1}<\cdots <j_{t}\leq N$, we denote
\[(i_{1},\cdots ,i_{s}|j_{1},\cdots ,j_{t})=\left \{
\begin{array}{lll}
0,&& ,\mathrm{if}\ \underline{i}\cap \underline{j}\neq \emptyset,\\
2\sharp \{(i_{k},j_{l})|i_{k}>j_{l}\}-st,&&\mathrm{otherwise}.
\end{array} \right.
\]
The \emph{q-flip} $\mathscr{T}=\bigoplus_{s,t} \mathscr{T}_{s,t}$:
$\bigwedge_\sigma(V)\otimes \bigwedge_\sigma(V)\rightarrow
\bigwedge_\sigma(V)\otimes \bigwedge_\sigma(V)$ is defined by: for
$1\leq i_{1}<\cdots <i_{s}\leq N$ and $1\leq j_{1}<\cdots
<j_{t}\leq N$,
$$\mathscr{T}_{s,t}(e_{i_{1}}\wedge \cdots \wedge e_{i_{s}}
\otimes e_{j_{1}}\wedge \cdots \wedge
e_{j_{t}})=(-q)^{(i_{1},\cdots ,i_{s}|j_{1},\cdots ,j_{t})}
e_{j_{1}}\wedge \cdots \wedge e_{j_{t}}\otimes e_{i_{1}}\wedge
\cdots \wedge e_{i_{s}}.
$$

Then $( \bigwedge_\sigma(V), \wedge, \mathscr{T})$ is a YB
algebra. Moreover it is twisted commutative.

\begin{lemma}Let $\sigma$ be a braiding on $V$ such that $\sigma^2=\mathrm{id}_V^{\otimes 2}$. Then the braiding $\beta$ on $T(V)$ also satisfies that $\beta^2=\mathrm{id}_{T(V)}^{\otimes 2}$.\end{lemma}
\begin{proof}We prove the statement for $\beta_{ij}$ by using induction on $i+j$.

When $i=j=1$, it is trivial since $\beta_{11}=\sigma$.

For $i+j\geq 3$, we have
\begin{eqnarray*}
\beta_{ji}\circ\beta_{ij}&=&(\beta_{j-1,i} \otimes
\mathrm{id}_V)(\mathrm{id}_V^{\otimes j-1}\otimes
\beta_{1i})(\mathrm{id}_V^{\otimes j-1}\otimes \beta_{i1})(\beta_{i,j-1}\otimes \mathrm{id}_V)\\
&=&\mathrm{id}_{T(V)}^{\otimes 2}.\end{eqnarray*}\end{proof}

If $\sigma=\pm \tau$, then $\sigma^2=\mathrm{id}_V^{\otimes 2}$.
For a general braiding $\sigma$, $\sigma^2$ is not necessarily the
identity map. The first nontrivial example where we nevertheless
have involution is the q-flip $\mathscr{T}$, i.e.,
$\mathscr{T}^2=Id$.

\begin{theorem}Let $(V,m, \sigma)$ be a YB algebra. Then the quantum quasi-shuffle algebra $(T^c(V),\Join_\sigma,\beta)$ is twisted commutative if and only if $(V,m, \sigma)$ is twisted commutative and $\sigma^2=\mathrm{id}_V^{\otimes 2}$.\end{theorem}
\begin{proof}If $(T^c(V),\Join_\sigma,\beta)$ is twisted
commutative, then on $V\underline{\otimes}V$ we have
\begin{eqnarray*}
m+\mathrm{id}_V^{\otimes 2}+\sigma&=&m+\mbox{\cyr sh}_\sigma\\
&=&\Join_{\sigma (1,1)}\\
&=&\Join_{\sigma (1,1)}\circ\sigma\\
&=&m\circ\sigma+\sigma+\sigma^2.
\end{eqnarray*}
By comparing the degrees, we have that $m=m\circ \sigma$ and
$\sigma^2=\mathrm{id}_V^{\otimes 2}$.

Conversely, we use induction on $i+j$ where $i$ and $j$ are the
powers of $V^{\otimes i}\underline{\otimes}V^{\otimes j}$.

When $i=j=1$, it is trivial.

For $i+j\geq 3$, by the inductive relation (1), we have
\begin{eqnarray*}
\lefteqn{\Join_{\sigma (j,i)}\circ \beta_{ij} }\\
&=&(\Join_{\sigma (j,i-1)} \otimes \mathrm{id}_V)(\beta_{i-1,j}\otimes \mathrm{id}_V)(\mathrm{id}_V^{\otimes i-1}\otimes \beta_{1,j})\\
&&+(\Join_{\sigma (j-1,i)}\otimes \mathrm{id}_V)(\mathrm{id}_V^{\otimes j-1}\otimes \beta_{1,i})(\mathrm{id}_V^{\otimes j-1}\otimes \beta_{i,1})(\beta_{i,j-1}\otimes \mathrm{id}_V)\\
&&+(\Join_{\sigma (j-1,i-1)}\otimes m)(\mathrm{id}_V^{\otimes
j-1}\otimes
\beta_{1,i-1}\otimes \mathrm{id}_V)\\
&&\ \ \ \ \ \circ(\mathrm{id}_V^{\otimes j-1}\otimes
\beta_{i-1,1}\otimes \mathrm{id}_V)(\mathrm{id}_V^{\otimes
i+j-2}\otimes
\beta_{1,1})(\beta_{i,j-1}\otimes \mathrm{id}_V)\\
&=&(\Join_{\sigma (i-1,j)} \otimes \mathrm{id}_V)(\mathrm{id}_V^{\otimes i-1}\otimes \beta_{1,j})\\
&&+(\Join_{\sigma (j-1,i)}\otimes \mathrm{id}_V)(\beta_{i,j-1}\otimes \mathrm{id}_V)\\
&&+(\Join_{\sigma (j-1,i-1)}\otimes m\circ \sigma)(\beta_{i,j-1}\otimes \mathrm{id}_V)\\
&=&\Join_{\sigma (i,j)}.
\end{eqnarray*} \end{proof}

\section{Basis coming from Lyndon words}
In this section, we use quantum quasi-shuffle products and Lyndon
words to present a new linear basis of the tensor space $T(V)$ for
a special kind of YB algebras $(V,m,\sigma)$. Let $(V,m,\sigma)$
be a finite dimensional YB algebra with linear basis
$(e_1,\ldots,e_N)$ and braiding of the following form:
$\sigma(e_i\otimes e_j)=q_{ij}e_j\otimes e_i$, where $q_{ij}$'s
are powers of a fixed nonzero scalar $q\in K$ and $q$ is not a
root of unity. For example, $( \bigwedge_\sigma(V), \wedge,
\mathscr{T})$ is certainly such a YB algebra.

$T^+(V)=T(V)/K$ always has a $K$-linear basis
$$\mbox{(I)}=\{e_{i_1}\otimes\cdots\otimes e_{i_m}|m>0,1\leq
i_1,\ldots,i_m\leq N\}.$$ The length of
$e_{i_1}\otimes\cdots\otimes e_{i_m}$ is $m$ and is denoted by
$|e_{i_1}\otimes\cdots\otimes e_{i_m}|=m$.

Given a total ordering on $e_i$'s, for example, say
$e_1<e_2<\cdots<e_N$, then there is a total ordering on (I)
provided by the lexicographic ordering, with the convention that
$a\leq a\otimes b$ for $a,b\in T^+(V)$. Lyndon words of $T^+(V)$
are defined as follows.
\begin{definition}
 An element $p$ in \emph{(I)} is called a \emph{Lyndon word} if, for any splitting $p=a\otimes b$, with $a,b\in \mbox{\emph{(I)}}$, one has $p<b$.
\end{definition}

Every $p$ in (I) has a unique factorization with respect to Lyndon
words. More precisely, $p$ can be written in a unique way  as a
tensor product of minimal number of Lyndon words (see \cite{Lot}).
We call this the standard factorization of $p$. In fact,
$p=p_1\otimes\cdots\otimes p_r$, where $p_i$'s are Lyndon words,
is the standard factorization of $p$ if and only if $p_1\geq
p_2\geq \cdots\geq p_r$. Denote the set of Lyndon words in  (I) by
$L$. Then let
$$\mbox{(I)}^\prime=\{l_1\otimes\cdots\otimes l_r~|~l_i\in L, l_1\geq \cdots\geq l_r\},$$
 we have (I)$=$(I)$^\prime$.

\begin{proposition}The set $$\mbox{\emph{(II)}}=\{l_1\Join_\sigma\cdots\Join_\sigma l_r~|~l_i\in L,l_1\geq\cdots\geq
l_r\}$$ forms a $K$-linear basis of $T^+(V)$.
\end{proposition}
\begin{proof}
First we note that
  $(T^c(V),\Join_\sigma)$ is a filtered algebra with
  $$T^c(V)^{[n]}:=\bigoplus_{i=0}^{n} V^{\otimes i},~ T^c(V)^{[n]}\subsetneqq T^c(V)^{[n+1]},$$
  and
$$T^c(V)^{[m]}\Join_\sigma T^c(V)^{[n]}\subset T^c(V)^{[m+n]}.$$
And the quantum shuffle algebra $T_\sigma(V)$ is a graded algebra
with
$$T^{n}_\sigma(V)=V^{\otimes
n},~T_\sigma(V)=\bigoplus_{n=0}^{\infty}T_\sigma^{n}(V),$$and
$$T_\sigma^{m}(V)\mbox{\cyr sh}_\sigma T_\sigma^{n}(V)\subset T_\sigma^{m+n}(V).$$

 Moreover $T_\sigma(V)$ is the associated graded algebra
 of $(T^c(V),\Join_\sigma)$  with respect to the above filtration,
since for any $l_1\otimes\cdots\otimes l_r\in V^{\otimes n}$,
 $$ l_1\Join_\sigma\cdots\Join_\sigma
l_r=l_1\mbox{\cyr sh}_\sigma\cdots\mbox{\cyr sh}_\sigma l_r\quad
\mbox{mod}~ T^c(V)^{[n-1]}. $$

Hence (II) is a linear basis of $T^+(V)$ if and only if
$$\mbox{(III)}=\{l_1\mbox{\cyr sh}_\sigma\cdots\mbox{\cyr
sh}_\sigma l_r~|~l_i\in L,l_1\geq\cdots\geq l_r\}$$ is a basis of
$T^+(V)$. For $l_i\in L$ with $l_1\geq\cdots\geq l_r$, we have
$$l_1\mbox{\cyr sh}_\sigma\cdots\mbox{\cyr sh}_\sigma l_r=a l_1\otimes\cdots\otimes l_r+\sum_{
\begin{matrix}
a_w\in K, w\in \mbox{(I)},\\
w<l_1\otimes\cdots\otimes l_r
\end{matrix}
}a_w w,$$
 where $a$ is a  scalar. After collecting the same $l_i$'s, we rewrite
$l_1\otimes\cdots\otimes l_r=p_1^{\otimes n_1}\otimes\cdots\otimes
p_s^{\otimes n_s}$, where $p_i\mbox{'s}\in L\subset \mbox{(I)}$
and $p_1>\cdots>p_s$. Set $p_i=e_{j_1}\otimes\cdots\otimes
e_{j_{m_i}}$ and $Q_i=\prod_{k,l\in \{j_1,\cdots,j_{m_i}\}}q_{kl}$
for $1\leq i\leq s$. Then $a=(n_1)_{Q_1}!\cdots (n_s)_{Q_s}!$,
where$(n)_\nu=\frac{\nu^n-1}{\nu-1}$ and
$(n)_\nu!=(n)_\nu(n-1)_\nu\cdots (1)_\nu$. By the requirements for
$q_{ij}$, we have that $a$ never vanishes. Hence the
transformation matrix from (I)' to (III) is triangular with
nonzero entries on its main diagonal, which
 implies that (III) is a basis of $T^+(V)$.
\end{proof}

\section*{Acknowledgements}The authors would like to thank the
referees for useful comments and suggestions which improved the
presentation of this paper. R.-Q. Jian was partially supported by
the China-France Mathematics Collaboration Grant 34000-3275100
from Sun Yat-sen University.

\end{document}